\documentclass{amsart}
\usepackage{amsaddr}
\usepackage{amssymb}
\usepackage{graphicx}
\usepackage{algorithm2e}
\usepackage{times}
\usepackage[all,cmtip]{xy}
\usepackage{tikz-cd}
\usepackage[utf8]{inputenc}
\usepackage{mathrsfs}
\usepackage{enumerate}

\DeclareSymbolFont{bbold}{U}{bbold}{m}{n}
\DeclareSymbolFontAlphabet{\mathbbold}{bbold}

\newcommand{\llangle}{\langle \! \langle}
\newcommand{\rrangle}{\rangle \! \rangle}
\newcommand{\D}{\mathbf{d}}
\newcommand{\R}{\mathbf{r}}

\def\st{\operatorname{star}}

\def\Z{{\mathbb Z}}

\def\curlyR{{\mathcal R}}

\def\curlyP{{\mathcal P}}

\def\P{{\mathcal P}}
\def\ra{\rightarrow}

\def\st{\operatorname{star}}

\def\varep{\varepsilon}

\def\<{\langle}
\def\>{\rangle}
\def\leq{\leqslant}

\def\Sq{\operatorname{Sq}^{\rho}}
\def\coker{\operatorname{coker}}

\def\dom{\mathbf{d}}
\def\ran{\mathbf{r}}

\def\hat{\widehat}
\def\st{\operatorname{star}}

\makeatletter
\newcommand{\eqnum}{\refstepcounter{equation}\textup{\tagform@{\theequation}}}
\makeatother

\newtheorem{thm}{Theorem}[section]
\newtheorem{prop}[thm]{Proposition}
\newtheorem{lemma}[thm]{Lemma}
\newtheorem{cor}[thm]{Corollary}
\newtheorem*{thm*}{Theorem}
\newtheorem*{cor*}{Corollary}

\theoremstyle{plain} 
\newcommand{\thistheoremname}{}
\newtheorem{genericthm}[thm]{\thistheoremname}

\theoremstyle{definition}
\newtheorem{defn}{Definition}[section]
\newtheorem*{defn*}{Definition}
\newtheorem*{prop*}{Proposition}
\newtheorem{ex}[thm]{Example}
\newtheorem*{ex*}{Example}
\newtheorem*{rem*}{Remark}

\numberwithin{equation}{section}

\begin{document}

\title[Groupoids and rewriting]{Groupoids and the algebra of rewriting in group presentations}

\author{N.D. Gilbert and E.A. McDougall}
\address{Department of Mathematics and the Maxwell Institute for the Mathematical Sciences, Heriot-Watt University, Edinburgh, EH14 4AS}
\email{N.D.Gilbert@hw.ac.uk}
\email{eam16@hw.ac.uk}






\maketitle

\begin{abstract}
Presentations of groups by rewriting systems (that is, by monoid presentations), have been fruitfully studied
by encoding the rewriting system in a $2$--complex --
the Squier complex --
whose fundamental groupoid then describes the derivation of consequences of the rewrite rules.
We describe a reduced form of the Squier complex, investigate the structure of its fundamental groupoid, and show that key properties of the presentation are still encoded in the reduced form.
\end{abstract}



\section*{Introduction}
\label{intro}
The study of the relationships between presentations of semigroups, monoids, and groups, and systems of rewriting rules has drawn together concepts from group and semigroup theory, low-dimensional
topology, and theoretical computer science.  In \cite{Sq}, Squier addressed the question of whether a 
finitely presented monoid with solvable word problem is necessarily presented by a finite, complete,
string rewriting system. He proved that a monoid presented by a finite, complete, string rewriting system must
satisfy the homological finiteness condition $FP_3$: indeed, an earlier result of Anick \cite{An} implies that
such a monoid satisfies the stronger condition $FP_{\infty}$.  These ideas are concisely surveyed by Cohen \cite{Coh},
and more extensively by Otto and Kobayashi in \cite{OK}.
Since examples are known of finitely presented monoids with solvable word problem that do not satisfy $FP_3$,
Squier's work shows that such monoids need not be presented by finite, complete, string rewriting systems.

Squier, Otto and Kobayashi \cite{SOK} studied finite, complete, string rewriting systems for monoids and proved that
the existence of such a system presenting a monoid $M$ implies a homotopical property -- {\em finite
derivation type} -- defined for a graph that encodes the rewriting system.  Moreover, they show that having finite derivation type 
does not depend on the particular rewriting system used to present $M$, and so is a property of $M$ itself and
a necessary condition that $M$ should be presented by a finite, complete string rewriting system.  

Finite derivation type is naturally thought of as a property of a $2$--complex, the {\em Squier complex}  associated to a monoid presentation $\P$, 
and obtained by adjoining certain $2$--cells to the graph of \cite{SOK}.  This point of view was introduced independently by Pride \cite{Pr1} and 
Kilibarda \cite{Ki}, and then extensively developed by Guba and Sapir in terms of both string-rewriting systems \cite{GS1} and more
geometrically, in terms of directed $2$-complexes \cite{GS2}.  The theory developed by Kilibarda and then by Guba and Sapir focusses on
the properties of {\em diagram groups}, which are fundamental groups of the Squier complex.  

Kilibarda \cite{Ki} studied the fundamental groupoid  of the Squier complex associated to a
monoid presentation $[X:R]$.  Gilbert \cite{Gi1} showed that the fundamental groupoid  is a monoid in the category of groupoids, and used this enriched structure to explain Pride's corresponding theory of diagram groups for monoid presentations of groups \cite{Pr2}. 

Pride's approach in \cite{Pr2} is based upon the addition of extra $2$--cells to a Squier complex so as to realise
a homotopy relation introduced by Cremanns and Otto \cite{CrOt}.  This
augmented Squier complex was called the \textit{Pride complex} in \cite{Gi1} and denoted by $K^+$.  Beginning with
a group presentation $\P = \< X : R \>$ of a group $G$,  we obtain a monoid presentation of $G$ by adding relations
$xx^{-1}=1=x^{-1}x$ for each $x \in X$, and the additional $2$--cells correspond to possible overlaps in the 
use of such relations in the free reduction of words on $A= X \cup X^{-1}$.  The outcome is that if $u$ and $v$ are
freely equivalent then any two edge-paths in $K^+$ from $u$ to $v$ that record this free equivalence are fixed-end-point
homotopic, as required for a homotopy relation as defined in \cite{CrOt}.  Gilbert investigated the structure of the fundamental groupoid $\pi(K^+,A^*)$ and showed that there is
a retraction map $\pi(K^+,A^*) \ra \pi(K^+,F(X))$ to the fundamental groupoid with vertex set the free group
$F(X)$.  

In this paper -- which reconfigures the approach to monoid presentation of groups in \cite{CrOt} and is  a somewhat belated sequel to \cite{Gi1,Pr2} -- we adopt a similar approach, but use a different modification of the Squier complex,
defining the \textit{reduced Squier complex} $\Sq(\P)$ of a group presentation $\P = \< X : R \>$ as a $2$--complex having  vertex set $F(X)$.
We can then work directly with the fundamental groupoid $\pi(\Sq(\P),F(X))$ and so avoid some of the technicalities from
\cite{Gi1}.  In particular, we show that the set $\st_1(\Sq(\P))$ of homotopy classes of paths in $\pi(\Sq(\P),F(X))$ that begin  at $1 \in F(X)$ has a natural group structure, and the end-of-path map
\[ \R : \st_1(\Sq(\P)) \ra F(X) \]
is a crossed module, as defined by J.H.C. Whitehead (see \cite{Wh2}).  We give a presentation for $\st_1(\Sq(\P))$,
and use it to show that the crossed module is isomorphic to that usually associated to a group presentation,
as in \cite{BrHu}.  It then follows that the fundamental group $\pi_1(\Sq(P),1)$ can be interpreted as the 
kernel of a free presentation of the relation module of $\P$, and as in \cite{CrOt} we may link the 
module structure of $\pi_1(\Sq(P),1)$ to the homological finiteness condition $\text{FP}_3$, and as in \cite{Dy} to
Cockcroft properties of $\P$.

A version of these results is presented in the second author's PhD thesis at Heriot-Watt University, Edinburgh.  The generous
financial support of a PhD Scholarship from the Carnegie Trust for the Universities of Scotland is duly and gratefully acknowledged.

\section{Background notions and notation}
\subsection{Groupoids}
\label{ord_gpds}
A {\em groupoid} $G$ is a small category in which every morphism
is invertible.  We consider a groupoid as an algebraic structure
as in \cite{HiBook}: the elements are the morphisms, and
composition is an associative partial binary operation.
The set of vertices of $G$ is denoted $V(G)$, and for each vertex $x \in V(G)$ there exists an identity morphism $1_x$.
An element $g \in G$ has domain
$g \dom$ and range $g\ran$ in $V(G)$, with $gg^{-1} = 1_{g \dom}$ and $g^{-1}g = 1_{g \ran}$.  
For $e \in V(G)$ the {\em star} of $e$ in $G$ is the set
$\st_{e}(G) = \{ g \in G : g \dom =e \}$, and the {\em local group} at $e$ is the set
$G(e) = \{ g \in G : g \dom = e = g \ran \}$.  

\subsection{Crossed modules}
\label{intro_cm}
Crossed modules will be the algebraic models of group presentations
that we shall use in our formulation of the relation module and the module of identities for a group presentation.  For a more detailed account of these topics , we refer to \cite{BrHu}.

A \textit{crossed module} is a group homomorphism $\partial : T \ra \Gamma$ together with an action of $\Gamma$ on $T$
(written $(t,g) \mapsto t^g$) such that $\partial$ is $\Gamma$--equivariant, so that for all $t \in T$ and $g \in \Gamma$ we have
\begin{equation} \label{cm1}
(t^g) \partial = g^{-1} (t \partial) g \,. 
\end{equation}
and such  that for all $t,u \in T$, we have:
\begin{equation} \label{cm2}
t^{u \partial} = u^{-1}tu \,. 
\end{equation}
We shall say that $(T, \partial)$ is a \textit{crossed $\Gamma$--module}.

\begin{ex}
Examples of crossed modules include the following:
\leavevmode
\begin{itemize}
\item any $\Gamma$--module $M$ with the trivial map $M \stackrel{0}{\ra} \Gamma$,
\item the inclusion of any normal subgroup $N \hookrightarrow \Gamma$,
\item the map $T \ra \operatorname{Aut} T$ that associates to $t \in T$ the inner automorphism of $T$ defined by
$a \mapsto t^{-1}at$,
\item any surjection $T \ra \Gamma$ with central kernel, where $\Gamma$ acts on $T$ by lifting and conjugation,
\item the boundary map $\pi_2(X,Y) \ra \pi_1(Y)$ from the second relative homotopy group of a pair of spaces $(X,Y)$
with $Y \subseteq X$.
\end{itemize}
\end{ex}
The last example motivated the introduction of  crossed modules  by J.H.C. Whitehead \cite{Wh2}.  

Let $\partial : T \ra \Gamma$ be a crossed module, and let $N$ be the image of $\partial$.  The following properties are easy consequences of
\eqref{cm1} and \eqref{cm2}.
\begin{itemize}
\item $N$ is normal in $\Gamma$, and so if we set $G= \Gamma /N$ we get the short exact sequence of groups:
$1 \rightarrow N \rightarrow \Gamma \rightarrow G \rightarrow 1 \,.$
\item $\ker \partial \subseteq Z(T)$, the center of $T$, so $\ker \partial$ is abelian.
\item $\ker \partial$ is invariant under the $\Gamma$--action on $T$, and so is a $\Gamma$--module.
\item $N$ acts trivially on $Z(T)$ and thus on $\ker \partial$, hence $\ker \partial$ inherits an action of $G$ to become a $G$--module.
\item the abelianisation $T^{ab}$ of $T$ inherits the structure of a $G$-module.
\end{itemize}

\begin{defn}
Let $(T, \partial)$ and $(T', \partial')$ be crossed $\Gamma$-modules. A morphism $\phi:(T, \partial) \rightarrow (T', \partial')$  is a group homomorphism  $\phi: T \rightarrow T'$ such that for $t \in T$, and $g \in \Gamma$, $ (t^g)\phi = (t \phi)^g$ and $\phi \partial' = \partial$.
\end{defn}

\subsubsection{Free crossed modules}
\label{free_cm_sec}
\begin{defn}
Let $(T, \partial)$ be a crossed $\Gamma$-module , let $R$ be a set, and let $\rho :R \rightarrow T$ be a function.
We say $(T, \partial)$ is a \textit{free crossed $\Gamma$-module with basis $\rho$} if for any crossed $\Gamma$-module $(T', \partial')$ and function $\sigma: R \rightarrow T'$ such that $\sigma \partial' = \rho \partial$, that is, such that the square
$$\xymatrix{
R \ar[d]_{\sigma} \ar[r]^{\rho} &T \ar[d]^{\partial}\\
T' \ar[r]_{\partial'} &\Gamma}$$
commutes, then there exists a unique morphism of crossed modules $\phi:(T, \partial) \rightarrow (T', \partial')$ such that $\rho \phi = \sigma$, that is,
\[ \xymatrixcolsep{3pc}\xymatrix{
R \ar[d]_{\sigma} \ar[r]^{\rho} & T \ar[d]^{\partial} \ar@{-->}[dl]_{\phi}\\
T' \ar[r]_{\partial'} &\Gamma}\]
also commutes.  We may also choose to emphasise $\omega = \rho \partial : R \rightarrow \Gamma$ by saying that a free crossed module $(T, \partial)$ with basis $\rho$ is a \textit{free crossed module on $\omega$}.
\end{defn}
The construction of free crossed modules is due to Whitehead \cite{Wh2}, and is also discussed in \cite{BrHu}.

\begin{prop}{\cite{Wh2}}
 \label{fre_cm_exist}
Let $\Gamma$ be a group, $R$ a set, and $\omega :R \rightarrow \Gamma$ a function. Then a free crossed $\Gamma$-module on $\omega$ exists and is unique up to isomorphism.
\end{prop}

\begin{proof}
We sketch the construction, following \cite[Proposition 5]{BrHu}.
Let $F$ be the free group on the basis $R \times \Gamma$.  Then $\Gamma$ acts on $F$ by right multiplication of basis 
elements: for $r \in R$ and $u,v \in \Gamma$ we have $(r,u)^v=(r,uv)$.  We map $(r,u) \mapsto u^{-1}(r \omega)u$ and this
induces a group homomorphism $\delta : F \ra \Gamma$.  The subgroup $P$ of $F$ generated by all elements of the form
\[ (r,u)^{-1}(s,v)^{-1}(r,u)(s,vu^{-1}(r\omega)u) \]
with $r,s \in R$ and $u,v \in \Gamma$ is normal in $F$, invariant under the $\Gamma$--action, and contained in 
$\ker \delta$.   It follows that $\delta$ induces $\partial : F/P \ra \Gamma$, and this is a free crossed $\Gamma$--module on
$\omega$.  Uniqueness up to isomorphism follows from the usual universal argument.
\end{proof}

Whitehead also observed the following:

\begin{prop}{\cite[page 457]{Wh2}}
\label{free_cm_free_mod}
Let $(C, \partial)$ be the free crossed $\Gamma$-module with basis $\rho$, and set $Q = \coker \partial$.  Then $C^{ab}$ is a 
free $Q$-module on the image of the composition $\overline{\rho}: R \xrightarrow{\rho} C \rightarrow C^{ab}$.
\end{prop}

\subsubsection{Crossed modules from group presentations}
\label{cm_from_pres}
A group \textit{presentation}  $\mathcal{P}= \< X : \mathcal{R} \>$ of a group $G$, consists of a set of \textit{generators} $X$, and a set of \textit{relations} $\mathcal{R} \subseteq (X \cup X^{-1})^* \times (X \cup X^{-1})^*$.  We set $A = X \cup X^{-1}$, we let $\rho : A^* \ra F(X)$
be the canonical map, and we define $\hat{\rho} : \curlyR \ra F(X)$ by $(\ell,r) \hat{\rho} = (\ell^{-1}r)\rho$.
We let $R$ be the image of $\hat{\rho}$ in $F(X)$, and define $N = \llangle R \rrangle$ to be
the normal closure of $R$ in $F$, so that  a typical element of $N$ has the form
\[ u_1^{-1}(r_1 \hat{\rho})^{\varep_1} u_1 \dotsm u_k^{-1}(r_k \hat{\rho})^{\varep_k}u_k \,, \] 
where, for $1 \leq j \leq k$, we have  $u_j \in F$, $r_j \in \curlyR$, and $\varep_j = \pm 1$.
Then $G$ is the quotient group $F(X)/N$, and we have a canonical \textit{presentation map} $\theta : F(X) \ra G$.

 We now let $(C({\mathcal P}),\partial)$ be the free crossed $F(X)$--module on the function $\hat{\rho} : {\mathcal R} \ra F(X)$. 
An element of $C = C({\mathcal P})$ is represented by a product
\[ (r_1,w_1)^{\varep_1} \dotsm (r_k,w_k)^{\varep_k} \]
where $r_j \in {\mathcal R}, w_j \in F(X)$ and $\varep_k = \pm 1$.  A typical Peiffer element (trivial in $C$) has the form
\[ (r,u)^{-1}(s,v)^{-1}(r,u)(s,vu^{-1}(r \hat{\rho})u) \,. \]
For $(r,w) \in C$ we have $\partial : (r,w) \mapsto w^{-1}(r \hat{\rho}) w$, and the image of $\partial$ is $N$.  We denote $\ker \partial$ by
$\pi =\pi({\mathcal P})$.  We therefore have short exact sequences of groups
\begin{equation}
\label{pres_ses}
1 \ra N \ra F(X) \ra G \ra 1 
\end{equation}
and
\begin{equation}
\label{free_cm_ses}
0 \ra \pi({\mathcal P}) \ra C({\mathcal P}) \ra F(X) \ra 1 \,,
\end{equation}
with $\pi$ central in $C$ and a $G$--module.  

\begin{prop}{\cite[Corollary to Proposition 7]{BrHu}} \label{pres_module} 
The free crossed module $C$ is isomorphic as a group to $\pi \times N$.  Its abelianisation $C^{ab}$ is a free 
$G$--module, and the induced map $\pi \ra C^{ab}$ is injective, so that 
we have a short exact sequence of $G$--modules.
\begin{equation} \label{pres_module_ses}
0 \rightarrow \pi \rightarrow C^{ab} \rightarrow N^{ab} \rightarrow 0
\end{equation}
\end{prop}

\begin{proof}
Since $F$ is free, \eqref{free_cm_ses} splits, and since $\pi$ is central in $C$ we have $C \cong \pi \times F$.  It follows that
$[C,C] \cong \{ 0 \} \times [F,F]$ and so $\pi \ra C^{ab}$ is injective. $C^{ab}$ is free by Proposition \ref{free_cm_free_mod}.
\end{proof}

In the sequence \eqref{pres_module_ses}, the $G$--module $N^{ab}$ is the \textit{relation module} of the presentation
$\curlyP$, and the $G$--module $\pi$ is the \textit{module of identities}.  The sequence \eqref{pres_module_ses} then gives
a free presentation of the relation module.

\section{Regular groupoids}
We now introduce some additional structure on a groupoid.  This idea originates in work of Brown and Gilbert \cite{BrGi},
and was further developed by Gilbert in \cite{Gi1} and by Brown in \cite{RBr}.  Brown uses the terminology
\textit{whiskered} groupoid for what Gilbert had called a \textit{semiregular} groupoid.  We shall use the semiregular
terminology, and will discuss in detail the special case of \textit{regular} groupoids.

\begin{defn} \label{reg_cond}
Let $G$ be a groupoid, with vertex set $V(G)$ and domain and range maps $\D, \R:G \rightarrow V(G)$. 
Then $G$ is \textit{semiregular} if
\begin{itemize}
\item $V(G)$ is a monoid, with identity $e \in V(G)$,
\item there are left and right actions of $V(G)$ on $G$, denoted $x \rhd \alpha$, $\alpha \lhd x$, which for all $x, y \in V(G)$ and $\alpha, \beta \in G$ satisfy:
\begin{enumerate}[(a)]
\item $(xy) \rhd \alpha = x \rhd (y \rhd \alpha)$; $\alpha \lhd (xy) = (\alpha \lhd x) \lhd y$; $(x\rhd \alpha) \lhd y = x \rhd (\alpha \lhd y)$,
\item $e \rhd \alpha = \alpha = \alpha \lhd e$,
\item $(x \rhd \alpha) \D = x(\alpha \D)$; $(\alpha \lhd x)\D = (\alpha \D)x$; $(x \rhd \alpha) \R = x(\alpha \R)$; $(\alpha \lhd x)\R = (\alpha \R)x$,
\item $x \rhd (\alpha \circ \beta) = (x \rhd \alpha) \circ (x \rhd \beta)$; $(\alpha \circ \beta) \lhd x = (\alpha \lhd x) \circ (\beta \lhd x)$, whenever $\alpha \circ \beta$ is defined,
\item $x \rhd 1_y = 1_{xy} = 1_x \lhd y$.
\end{enumerate}
\end{itemize}
A semiregular groupoid $G$ is a \textit{regular} groupoid if $V(G)$ is a group.
\end{defn}

Our first result collates some simple facts from \cite[section 1]{Gi1}.

\begin{prop}
 \label{star op} 
\leavevmode
\begin{enumerate}[(a)]
\item Let $G$ be a semiregular groupoid.  Then there are two everywhere defined binary operations on $G$ given by:
\begin{align*} \alpha * \beta &= (\alpha \lhd \beta \D) \circ (\alpha \R \rhd \beta) \\
\alpha \circledast \beta &= (\alpha \D \rhd \beta) \circ (\alpha \lhd \beta \R) \,.
\end{align*}
Each of the binary operations $*$ and $\circledast$ make $G$ into a monoid, with identity $1_e$.
\item The binary operation $*$ and the monoid structure on $V(G)$ make the semiregular groupoid $G$ into a strict monoidal groupoid if and only if the operations $*$ and $\circledast$ on $G$ coincide.
\item Let $G$ be a regular groupoid.   Then each of the two binary operations $\ast$ and $\circledast$ given in Proposition \ref{star op} make  $G$ into a group, with identity $1_e$.
\item Let $G$ be a regular groupoid. Then $\R : (G, \ast) \ra V(G)$ is a group homomorphism, and
$\st_e(G)$ is a subgroup of $(G,\ast)$.
\end{enumerate}
\end{prop}

\begin{proof}  We remark only on the proof of (c), since it is mis-stated in \cite{Gi1}.
The inverse of $\alpha$ with respect to $*$ is
$$\alpha^* = \alpha \R^{-1} \rhd \alpha^{\circ} \lhd \alpha \D^{-1}$$
and with respect to $\circledast$ is
$$\alpha^{\circledast} = \alpha \D^{-1} \rhd \alpha^{\circ} \lhd \alpha \R^{-1}$$
where $^{\circ}$ is the inverse of $\alpha$ with respect to the groupoid operation, and $^{-1}$ is the inverse in the group $V(G)$.
\end{proof}

\begin{defn}
In view of part (c) of Proposition \ref{star op}, we say that a semiregular groupoid is  \textit{monoidal} if
the operations $*$ and $\circledast$ coincide.  (Brown \cite{RBr} calls such semiregular groupoids \textit{commutative}.)
\end{defn}

Still following \cite[section 1]{Gi1}, we state the connection between regular groupoids and crossed modules.

\begin{prop}
\label{reg_gpd_is_cm}
In a regular groupoid $G$, the group $(G,\ast)$ admits a group
action of $V(G)$ by automorphisms, defined for $\alpha \in G$ and $q \in V(G)$ by $\alpha^q = q^{-1} \rhd \alpha \lhd q$.
Then $\R : \st_e(G) \ra V(G)$
is a crossed module if and only if $G$ is monoidal.  
\end{prop}

\section{The Squier complex of a group presentation}
Let $\curlyP= \langle X: \curlyR \rangle$ be a group presentation.  Recall from section \ref{cm_from_pres} that relations
$(l,r) \in \curlyR$ may involve words in $(X \cup X^{-1})^*$ that are not freely reduced.  However, to reduce notational
clutter, we shall suppress mention of the free reduction map $\rho : (X \cup X^{-1})^* \ra F(X)$ in what follows.  Hence
if $p,q \in F(X)$ and $(l,r) \in \curlyR$, we shall write $prq$ for $p(r \rho)q$, and so on.

\begin{defn}
The  \textit{reduced Squier complex} $\Sq(\P)$ is the $2$--complex defined as follows:
\begin{itemize}
\item the vertex set of $\Sq(\curlyP)$ is the free group $F(X)$ on $X$,
\item the edge set of $\Sq(\curlyP)$ consists of all 3-tuples $(p, l,r, q)$ with $p, q \in F(X)$ and $(l,r) \in \curlyR$. Such an edge will start at $plq$ and end at $prq$, so each edge corresponds to the application of a relation in $F(X)$. 
\item the 2-cells correspond to applications of non-overlapping relations, and so a 2-cell is attached along every edge path of the form:
$$
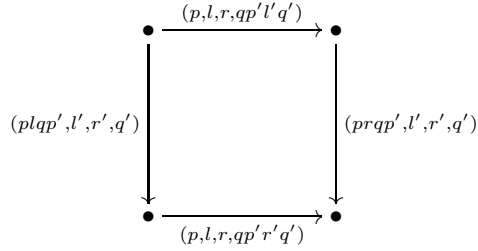
\xymatrix@=6em{
 \bullet \ar[d]_{(plqp',  l',r', q')} \ar[r]^{(p, l,r, qp'l'q')} & \bullet \ar[d]^{(prqp', l', r', q')} \\
 \bullet\ar[r]_{(p, l,r, qp'r'q')} & \bullet }$$
The edge paths 
\begin{align*}
& (p, l,r, qp'l'q')(prqp', l' , r', q') \\ \intertext{and} & (plqp', l' , r', q')(p, l,r, qp'r'q')
\end{align*}
 will therefore be homotopic in $\Sq(\P)$.
 \end{itemize}
\end{defn}

\begin{lemma}
The fundamental groupoid $\pi(\Sq(\P),F(X))$ of the Squier complex $\Sq(\P)$ of a group presentation $\P$ is a regular groupoid.
\end{lemma}

\begin{proof}
The vertex set of $\pi = \pi(\Sq(\P),F(X))$ is the group $F(X)$.
We need to define left and right actions of $F(X)$ on homotopy classes of paths in $\Sq(\P)$. We first define such actions for single edges.
Let $u, v \in F(X)$ and suppose that $(p, l,r, q)$ is an edge in $\Sq(\P)$.  We define
\begin{align}
u \rhd (p, l,r, q) &= (up, l,r, q) \\
(p, l,r, q) \lhd v &= (p, l,r, qv) \,.
\end{align}
It is then clear that these actions can be extended to edge-paths in $\Sq(\P)$, and induce actions of $F(X)$ on
homotopy classes of paths.
\end{proof}

In what follows it will be convenient to work directly with edge paths in $\Sq(\P)$, even though these are to be interpreted as
representatives of homotopy classes in the fundamental groupoid $\pi(\Sq(\P),F(X))$.  In particular, we shall apply the operations
$\ast$ and $\circledast$ directly to edge paths.

\begin{thm} \label{monoidal}
The regular groupoid $\pi(Sq(\curlyP),F(X))$ is monoidal.
\end{thm}

\begin{proof}
Recall that 
\begin{align*}
\alpha \ast \beta &= (\alpha \lhd \beta \D) \circ (\alpha \R \rhd \beta) \\
\alpha \circledast \beta &= (\alpha \D \rhd \beta) \circ (\alpha \lhd \beta \R) \,.
\end{align*}
First we consider single-edge paths $\alpha = (p,l,r,q)$ and $\beta=(p',l',r',q')$. Then
\begin{align*}
\alpha * \beta &= (p,l,r,qp'l'q')\circ(prqp',l',r',q')\\
\alpha \circledast \beta &= (plqp',l',r',q') \circ (p,l,r,qp'r'q') \,.
\end{align*}
These paths comprise the boundary of a  $2$--cell in $\Sq(\P)$ and are thus homotopic:  hence $\alpha * \beta = \alpha \circledast \beta$.

Now consider edge paths $\alpha = \alpha_1 \circ \alpha_2 \circ \dotsb \circ \alpha_k$ and $\beta = \beta_1 \circ \beta_2 \circ \dotsb \circ \beta_m$  and with each $\alpha_i$, $\beta_j$ single edges. We set $m=1$ and $k>1$: then we may assume that if $\beta$ is the single edge $\beta_1$ then
\[ (\alpha_1 \circ \dotsb \circ \alpha_{k-1}) \ast \beta_1 = (\alpha_1 \circ \dotsb \circ \alpha_{k-1}) \circledast \beta_1 \,. \]
We then have
\begin{align*}
\alpha * \beta_1 &= (\alpha \lhd \beta_1 \D) \circ (\alpha_k \R \rhd \beta_1)\\
&= (\alpha_1 \lhd \beta_1 \D) \circ (\alpha_2 \lhd \beta_1 \D) \circ \dotsb \circ (\alpha_k \lhd \beta_1 \D) \circ  (\alpha_k \R \rhd \beta_1)\\
&= (\alpha_1 \lhd \beta_1 \D) \circ \dotsb \circ (\alpha_{k-1} \lhd \beta_1 \D) \circ (\alpha_k * \beta_1)\\
&= (\alpha_1 \lhd \beta_1 \D) \circ \dotsb \circ (\alpha_{k-1} \lhd \beta_1 \D) \circ (\alpha_k \circledast \beta_1)\\
&= (\alpha_1 \lhd \beta_1 \D) \circ \dotsb \circ (\alpha_{k-1} \lhd \beta_1 \D) \circ (\alpha_{k}\D \rhd \beta_1) \circ (\alpha_k \lhd \beta_1 \R)\\
&= (\alpha_1 \lhd \beta_1 \D) \circ \dotsb \circ (\alpha_{k-1} \lhd \beta_1 \D) \circ (\alpha_{k-1}\R \rhd \beta_1) \circ (\alpha_k \lhd \beta_1 \R)\\
&=((\alpha_1 \circ \dotsb \circ \alpha_{k-1}) \ast \beta_1) \circ (\alpha_k \lhd \beta_1 \R)\\
&=((\alpha_1 \circ \dotsb \circ \alpha_{k-1}) \circledast \beta_1) \circ (\alpha_k \lhd \beta_1 \R)\\
&= (\alpha_1 \D \rhd \beta_1) \circ (\alpha_1 \lhd \beta_1 \R) \circ \dotsb \circ (\alpha_{k-1} \lhd \beta_1 \R) \circ (\alpha_k \lhd \beta_1 \R)\\
&= \alpha \circledast \beta_1
\end{align*}
So by induction on $k$, we have $\alpha \ast \beta = \alpha \circledast \beta$, whenever $m=1$. Now for $m>1$
we assume inductively that, for any edge path $\alpha$,
\[ \alpha \ast (\beta_1 \circ \dotsb \circ \beta_{m-1}) = \alpha \circledast (\beta_1 \circ \dotsb \circ \beta_{m-1}) \,.\]
Then
\begin{align*}
\alpha * \beta &= (\alpha \lhd \beta_1 \D) \circ (\alpha_n \R \rhd \beta)\\
&=\big((\alpha_1 \circ \dotsb \circ \alpha_n) \lhd \beta_1 \D\big) \circ \big(\alpha_n \R \rhd (\beta_1 \circ \dotsb \circ \beta_{m-1})\big) \circ \big(\alpha_n \R \rhd \beta_j\big)\\
&= \big( \alpha * (\beta_1 \circ \dotsb \circ \beta_{m-1}) \big) \circ \big(\alpha_n \R \rhd \beta_m\big)\\
&= \big( \alpha \circledast (\beta_1 \circ \dotsb \circ \beta_{m-1}) \big) \circ \big(\alpha_n \R \rhd \beta_m\big)\\
&= \big( \alpha_1 \D \rhd (\beta_1 \circ \dotsb \circ \beta_{m-1}) \big) \circ \big( \alpha \lhd \beta_{m-1}\R \big) \circ \big(\alpha_n \R \rhd \beta_j\big)\\
& = \big( \alpha_1 \D \rhd (\beta_1 \circ \dotsb \circ \beta_{m-1}) \big) \circ \big( \alpha \lhd \beta_m\D \big) \circ \big(\alpha_n \R \rhd \beta_j\big)\\
& = \big( \alpha_1 \D \rhd (\beta_1 \circ \dotsb \circ \beta_{m-1}) \big) \circ \big( \alpha * \beta_m \big)\\
& = \big( \alpha_1 \D \rhd (\beta_1 \circ \dotsb \circ \beta_{m-1}) \big) \circ \big( \alpha \circledast \beta_m \big)\\
& = \big( \alpha_1 \D \rhd (\beta_1 \circ \dotsb \circ \beta_{m-1}) \big) \circ \big( \alpha_1 \D \rhd \beta_m \big) \circ \big(\alpha \lhd \beta_m \R \big)\\
&= \big( \alpha \D \rhd \beta \big) \circ \big( \alpha \rhd \beta \R \big)\\
&= \alpha \circledast \beta
\end{align*}
Thus by induction  we have that $\alpha * \beta = \alpha \circledast \beta$, for all edge paths $\alpha, \beta$ in $\Sq(\P)$.
\end{proof}

From Proposition \ref{reg_gpd_is_cm} we have:

\begin{cor}
\label{squier_cm}
The subset $\st_1(\pi(\Sq(\P),F(X))$ of the fundamental groupoid of the Squier complex $\Sq(\P)$ is a group
under the binary operation $\ast$,
and the restriction of the range map is a crossed module 
\[ \R: \st_1(\pi(\Sq(\P),F(X))) \ra F(X) \,.\]
\end{cor}

\subsection{The crossed module of a Squier complex}
\label{cm_Sq_gp}
Our aim is now to show that the crossed module in Corollary \ref{squier_cm} is isomorphic to the free crossed module 
$C \xrightarrow{\partial} F(X)$ derived from the presentation $\P$, as in Section \ref{cm_from_pres}.  Furthering our
blurring of the distinction between an edge path and its homotopy class in the fundamental groupoid, we shall abbreviate
the group $\st_1(\pi(\Sq(\P),F(X)))$ as $\st_1(\Sq(\P))$.  We denote by $S_1$ the set of all edges $e \in \Sq(\P)$ with
$e \D = 1$, that is
\begin{align*} S_1 &= \{ (p,l,r,q) : p,q \in F(X), (l,r) \in \curlyR, plq=1 \} \\ &= \{ (q^{-1}l^{-1},l,r,q) : q \in F(X), (l,r) \in \curlyR \} \,.\end{align*}
We shall denote the edge $(q^{-1}l^{-1},l,r,q)$ by $\lambda_{l,r,q}$.

Let $e = (p,l,r,q)$ be an edge of $\Sq(\P)$ in the connected component of $1 \in F(X)$, and define
\[ e\lambda = (e\D)^{-1} \rhd e = \lambda_{l,r,q} \in S_1 \,.\]

\begin{prop} \label{lambda}
Let $\alpha$ be an edge path in $\st_1(\Sq(\curlyP))$.  Then $\alpha$  is equal to a $\ast$-product of single edges in $S_1$. 
Thus the group $(\st_1(\Sq(\curlyP)),\ast)$ is generated by $S_1$
\end{prop}

\begin{proof}
The claim is trivial for edge paths $\alpha$ of length $1$, so now suppose that
$$\alpha = \alpha_1 \circ \alpha_2 \circ \dotsb \circ \alpha_n$$
for some $n>1$, with each $\alpha_i $ a single edge.
Set $\lambda_i = \alpha_i \lambda=(\alpha_i\D)^{-1} \rhd \alpha_i$. Then  $\lambda_i \in S_1$, and $\alpha_1 = \lambda_1$.
We now assume inductively that
$$\alpha_1 \circ \alpha_2 \circ \dotsb \circ \alpha_{n-1} = \lambda_1 \ast \lambda_2 * \dotsb \ast \lambda_{n-1} \,.$$
Then
\begin{align*}
\alpha &= (\alpha_1 \circ \dotsb \circ \alpha_{n-1}) \circ \alpha_n\\
&= (\alpha_1 \circ \dotsb \circ \alpha_{n-1}) \circ (\alpha_n\D \rhd \lambda_n)\\
&=(\alpha_1 \circ \dotsb \circ \alpha_{n-1}) \circ (\alpha_{n-1}\R \rhd \lambda_n)\\
&=(\alpha_1 \circ \dotsb \circ \alpha_{n-1}) \ast \lambda_n\\
&=\lambda_1 \ast \lambda_2 \ast \dotsb * \lambda_{n-1} \ast \lambda_n \,.
\end{align*}
Therefore $\alpha = \lambda_1 \ast \dotsb \ast \lambda_n$ . 
\end{proof}

\begin{defn}
\label{lamdba_rewriting}
We denote the product $\lambda_1 * \dotsb \ast \lambda_n$ used to represent $\alpha \in \st_1(\Sq(\curlyP))$ in Proposition
\ref{lambda} by $\alpha \lambda$.
\end{defn}

\begin{lemma}
\label{lambda_product}
Suppose that $\alpha \circ \beta \in \st_1(\Sq(\P))$.  Then $(\alpha \circ \beta) \lambda = \alpha \lambda \ast \beta \lambda$.
\end{lemma}

We now want to understand the effect of homotopy of edge paths in $\Sq(\P)$ on the $\ast$--products defined in
Proposition \ref{lambda}.  We first consider a $1$--homotopy, that is, the insertion of deletion of a pair of inverse
edges.  Let $\xi = \rho \circ \sigma$  in $Sq(\curlyP)$, with $\rho \in \st_1(\Sq(\curlyP))$. Then consider the homotopic 
path $\xi' = \rho \circ \alpha \circ \alpha^{\circ} \circ \sigma$, with $\alpha$ a single edge. Then
\begin{align*}
\xi' \lambda &= \rho \lambda \ast \alpha \lambda \ast(\alpha^{\circ}) \lambda\ast \sigma \lambda\\
&= \rho \lambda \ast [( \alpha \D)^{-1} \rhd \alpha  \ast (\alpha^{\circ} \D)^{-1} \rhd \alpha^{\circ} ] \ast \sigma \lambda\\
&= \rho \lambda \ast [( \alpha \D)^{-1} \rhd \alpha  \ast (\alpha \R)^{-1} \rhd \alpha^{\circ} ] \ast \sigma \lambda\\
&= \rho \lambda \ast [ ( \alpha \D)^{-1} \rhd \alpha \lhd 1) \circ  (\alpha \D)^{-1} \alpha \R \rhd  ((\alpha \R)^{-1} \rhd \alpha^{\circ} )] \ast \sigma \lambda\\
&= \rho \lambda \ast [ ( \alpha \D)^{-1} \rhd \alpha) \circ  (\alpha \D)^{-1}  \rhd \alpha^{\circ} )] \ast \sigma \lambda\\
&= \rho \lambda \ast \sigma \lambda = \xi \lambda \,.
\end{align*}
Therefore a $1$--homotopy applied to an edge path $\xi$ does not change the $\ast$--product $\xi \lambda$.

Suppose that we have a $2$--cell
\begin{equation}
\label{general_2-cell}
\xymatrix@=6em{
\bullet  \ar[d]_{(plqt,  s,d, u)} \ar[r]^{(p, l,r, qtsu)} & \bullet \ar[d]^{(prqt, s,d, u)} \\
\bullet \ar[r]_{(p, l,r, qtdu)} & \bullet }
\end{equation}
in the connected component of $1 \in F(X)$ in $\Sq(\P)$, with
\begin{equation}
\label{naming_edges}  \alpha =(p, l,r, qtsu), \beta=(prqt,s,d,u), \gamma=(plqt,s,d,u), \delta=(p, l,r, qtdu) \,. \end{equation}
This $2$--cell gives a homotopy between $\alpha \circ \beta$ and $ \gamma \circ \delta$, or equivalently tells us that in
$\pi(\Sq(\P))$ we have
\[ (p,l,r,q) \ast (t,s,d,u) = (p,l,r,q) \circledast (t,s,d,u) \,.\]

If this $2$--cell is involved in a $2$--homotopy between edge paths $\xi$ and $\xi'$, we may assume using $1$--homotopies where necessary, that we have $\xi = \rho \circ \alpha \circ \beta \circ \sigma$ and 
$\xi' = \rho \circ \gamma \circ \delta \circ \sigma$, that is a configuration
$$
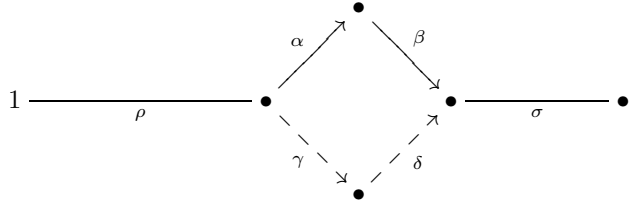
\xymatrix { 
& & & & \bullet \ar[dr]^{\beta}& & & \\
1 \ar@{-}[rrr]_{\rho}& & & \bullet \ar[ur]^{\alpha} \ar@{-->}[dr]_{\gamma} & & \bullet  \ar@{<--}[dl]^{\delta} \ar@{-}[rr]_{\sigma} & & \bullet\\
& & & & \bullet & & & }$$
Then, using $\simeq$ to denote homotopy of edge paths in $\Sq(\P)$, we have
\begin{equation}\label{2-cell_lambda} \begin{aligned}
\xi \lambda &= \rho \lambda \ast \alpha \lambda \ast \beta \lambda \ast \sigma \lambda\\
&= \rho \lambda \ast (\alpha \circ \beta) \ast \sigma \lambda\\
& \simeq  \rho \lambda \ast (\gamma \circ \delta) \ast \sigma \lambda\\
&=  \rho \lambda \ast \gamma \lambda \ast \delta \lambda \ast \sigma \lambda\\
&= \xi' \lambda \,.
\end{aligned} \end{equation}

The above considerations show that, for a given homotopy class in $\st_1(\Sq(\P))$, we may select a representative edge
path $\xi$ in the form of its $\ast$--product $\xi \lambda$ and that this product will be unique up to changes induced by the
$2$--cells in $\Sq(\P)$, which may modify the product as in equations \eqref{2-cell_lambda} above.  We can be more precise.

\begin{prop}
\label{def_rels_lam_prop}
Given $q \in F(X)$ and $(l,r) \in \curlyR$, we set
\[ \lambda_{l,r,q} = (q^{-1}l^{-1},l,r,q) \in S_1 \,.\]
Then the following are a set of defining relations for the group $(\st_1(\Sq(\P)),\ast)$ on the generating set $S_1$:
\begin{equation}
\label{def_rels_lam} \lambda_{l,r,vsu} \ast \lambda_{s,d,u} = \lambda_{s,d,u} \ast \lambda_{l,r,vdu} \,,\end{equation}
where $(l,r), (s,d) \in \curlyR$ and $u,v \in F(X)$.
\end{prop}

\begin{proof}
Since  
\[ \lambda_{l,r,vsu} \ast \lambda_{s,d,u} = (u^{-1}s^{-1}v^{-1}l^{-1},l,r,vsu) \circ (u^{-1}s^{-1}v^{-1}l^{-1}rv,s,d,u) \]
and
\[  \lambda_{s,d,u} \ast \lambda_{l,r,vdu} = (u^{-1},s,d,u) \circ (u^{-1}v^{-1}l^{-1},l,r,vdu) \,,\]
we see that the stated relations are true in $(\st_1(\Sq(\P)),\ast)$  since they record the equality of the two paths around the
sides of the $2$--cell
$$\xymatrix@=10em{
\bullet \ar[d]_{(u^{-1},s,d,u)} \ar[r]^{(u^{-1}s^{-1}v^{-1}l^{-1},l,r,vsu)} & \bullet \ar[d]^{(u^{-1}s^{-1}v^{-1}l^{-1}rv,s,d,u)} \\
\bullet \ar[r]_{(u^{-1}v^{-1}l^{-1},l,r,vdu)} & \bullet}$$
On the other hand, to accomplish the rewriting in \eqref{2-cell_lambda}, we need to identify the paths around the boundary of a
general $2$--cell as in \eqref{general_2-cell} and, in the notation of \eqref{naming_edges}, use the relation
\[ \alpha \lambda \ast \beta \lambda = \gamma \lambda \ast \delta \lambda \,. \]
Now
\begin{align*}
\alpha \lambda &= (u^{-1}s^{-1}t^{-1}q^{-1}l^{-1},l,r,qtsu) =(l^{-1},l,r,1)^{qtsu} \,, \\
\beta \lambda &=  (u^{-1}s^{-1},s,d,u) = (s^{-1},s,d,1)^u = \gamma \lambda\\
\intertext{and}
\delta \lambda &= (u^{-1}d^{-1}t^{-1}q^{-1}l^{-1},l,r,qtdu) = (l^{-1},l,r,1)^{qtdu} \,.
\end{align*}
If we set $v=qt$ then
\begin{align*}
\alpha \lambda &= (u^{-1}s^{-1}v^{-1}l^{-1},l,r,vsu) = \lambda_{l,r,vsu} \,, \\
\beta \lambda &=  \lambda_{s,d,u} = \gamma \lambda\\
\intertext{and}
\delta \lambda &= (u^{-1}d^{-1}v^{-1}l^{-1},l,r,vdu) = \lambda_{l,r,vdu} \,.
\end{align*}
and the required relation is
\[ \lambda_{l,r,vsu} \ast \lambda_{s,d,u} = \lambda_{s,d,u} \ast  \lambda_{l,r,vdu} \,. \]
\end{proof}

\begin{thm}
The crossed $F(X)$-module $\st_1(\Sq(\curlyP)) \xrightarrow{\R} F(X)$ derived from the Squier complex $\Sq(\curlyP)$ of a group
presentation $\P = \< X: \curlyR \>$, is isomorphic to the free crossed $F(X)$-module $C \xrightarrow{\partial} F$ derived from $ \curlyP$.
\end{thm}

\begin{proof}
Recall from section \ref{cm_from_pres} that the free crossed module $C \xrightarrow{\partial} F$
has basis function $v: \curlyR \rightarrow C$, $v: (l,r) \mapsto (l,r, 1)$.  We define
$\overline{v}: \curlyR \rightarrow \st_1(\Sq(\curlyP))$ by $\overline{v}: (l,r) \mapsto (l^{-1}, l,r, 1)$.  Then $v\partial = \overline{v}\R
$, and thus by freeness of $(C, \partial)$, we have a crossed module morphism $\phi: C \rightarrow \st_1(\Sq(\curlyP))$, defined on generators by $(l,r, u) \mapsto (u^{-1}l^{-1}, l,r, u) = \lambda_{l,r,u}$.  We note that this is a bijection from the group generating set of $C$ to
$S_1$.

To obtain an inverse to $\phi$, we therefore wish to map $\lambda_{l,r,u} \mapsto (l,r,u)$.  This will be well-defined and a 
homomorphism if and only if the defining relations given in \eqref{def_rels_lam}  in Proposition \ref{def_rels_lam_prop} are mapped to an equation that holds in the group $C$.  Now the left-hand side of \eqref{def_rels_lam} maps to
\[  (l,r,vsu)(s,d,u) \]
and the right-hand side to 
\[ (s,d,u)(l,r,vdu) \,. \]
and in the crossed $F(X)$--module $C$ we do indeed have
\[ (s,d,u)^{-1}(l,r,vsu)(s,d,u) = (l,r,vsu(u^{-1}s^{-1}du)) = (l,r,vdu) \,. \]
\end{proof}

The kernel of the map $\R: \st_1(\Sq(\curlyP)) \ra F(X)$ is the local group at $1 \in F(X)$ of the groupoid
$\pi(\Sq(\P), F(X))$, that is the fundamental group $\pi_1(\Sq(\P),1)$.  Then from Proposition \ref{pres_module} we
obtain:

\begin{prop} \label{Sq_rel_mod} 
Let $\P = \< X: \curlyR \>$ be a presentation of a group $G$ with presentation map $\theta : F(X) \ra G$ and let
$N = \ker \theta$, so that $N^{ab}$ is the relation module of $\P$.
Then we have a short exact sequence of $G$--modules:
\begin{equation} \label{Sq_ses}
0 \rightarrow \pi_1(\Sq(\P),1) \rightarrow \bigoplus_{r \in \curlyR} \Z G \rightarrow N^{ab} \rightarrow 0 \,.
\end{equation}
\end{prop}

\begin{ex}
Let $\P =\< x : xx^{-1}=1 \>$ presenting the infinite cyclic group $\< x \>$.  Then the relation modle is trivial,
and \eqref{Sq_ses} reduces to an isomorphism $\pi_1(\Sq(\P),1) \cong \Z \<x \>$.   We can also see this from
the construction of $\Sq(\P)$.  The Squier complex $\Sq(\P)$
has vertex set $\< x \>$ and each edge is a loop.  The generating set $S_1$ in Proposition \ref{lambda}
is
\[ S_1 = \{ (x^{-q},xx^{-1},1,x^q) : q \in \Z \} \]
and we write $\lambda_q = (x^{-q},xx^{-1},1,x^q)$.  By Proposition \ref{def_rels_lam_prop} we have a presentation
for  $\pi_1(\Sq(\P),1) = \st_1(\Sq(\P))$ given by
\[ \pi_1(\Sq(\P),1)  = \< \lambda_q \; (q \in \Z) : \lambda_{p+q} \ast \lambda_q = \lambda_q \ast
\lambda_{p+q} \; (p,q \in \Z) \> \]
and so $\pi_1(\Sq(\P),1)$ is free abelian of countably infinite rank, and the $\< x \>$--action is defined by
$\lambda_q^x = \lambda_{q+1}$.
\end{ex}

\subsection{Properties of $\pi_1(\Sq(\P),1)$}
We show in two Corollaries of Proposition \ref{Sq_rel_mod} how properties of the presentation $\P$ and the group $G$ are reflected in properties of the fundamental group of the reduced Squier complex.  The illustrative examples that we give 
are drawn from \cite{CrOt} and \cite{Dy}.

The first result was proved for the Squier complex of \cite{SOK} by Cremanns and Otto.  We refer to
\cite[Chapter VIII]{BrBook} and to \cite[section 4]{CrOt} for information on the condition $\text{FP}_3$

\begin{cor}{\cite[Theorem 4.10]{CrOt}}
Let $G$ be presented by the finite presentation $\P = \< X : \curlyR \>$.  Then the following are equivalent.
\leavevmode
\begin{enumerate}[(a)]
\item $\pi_1(\Sq(\P),1)$ is a finitely generated $G$--module,
\item $G$ is of type $\text{FP}_3$.
\end{enumerate}
\end{cor}

\begin{proof}
There is an exact sequence of $G$--modules (see \cite[Proposition II.5.4]{BrBook}),
\[ 0 \ra N^{ab} \ra \bigoplus_{x \in X} \Z G \rightarrow \Z G \ra \Z \ra 0 \,, \]
and if $\pi_1(\Sq(\P),1)$ is a finitely generated as a $G$ module by a set $S$ this extends, using \eqref{Sq_ses}, 
to a partial free resolution of finite type
\[ \bigoplus_S \Z G \rightarrow \bigoplus_{\curlyR} \Z G \rightarrow \bigoplus_X \Z G \rightarrow \Z G \ra \Z \ra 0 \,,\]
which shows that $G$ has type $\text{FP}_3$.  Conversely, if $G$ has type $\text{FP}_3$ then
$\pi_1(\Sq(\P),1)$ is the kernel (at dimension $2$) in a partial free resolution of $\Z$ of finite type and so
is finitely generated as a consequence of the generalized Schanuel Lemma, see \cite[Proposition 4.3]{BrBook}.
\end{proof}

The second result characterizes the Cockcroft properties of $\P$.  
Following Dyer \cite[Theorem 4.2]{Dy} we make the following definition.  Let $L$ be a subgroup of $G$,
and apply the tensor product $- \otimes_{L} \Z$ to the sequence \eqref{Sq_ses} to obtain the sequence
\begin{equation} \label{Cft_ses}
\pi_1(\Sq(\P),1) \otimes_L \Z \rightarrow \bigoplus_{r \in \curlyR} \Z (G/L) \rightarrow N^{ab} \otimes_L \Z \rightarrow 0 \,,
\end{equation}
of abelian groups, where $G/L$ is the set of left cosets of $L$ in $G$.  Then $\P$ is \textit{$L$--Cockcroft} if the
map 
\[\bigoplus_{r \in \curlyR} \Z (G/L) \rightarrow N^{ab} \otimes_L \Z \]
in \eqref{Cft_ses} is an isomorphism.  Immediately from \eqref{Cft_ses} we obtain part of \cite[Theorem 4.2]{Dy}
in terms of $\pi_1(\Sq(\P),1)$.

\begin{cor}{\cite[Theorem 4.2]{Dy}}
Let $G$ be presented by the presentation $\P = \< X : \curlyR \>$ and let $L$ be a subgroup of $G$.
Then the following are equivalent:
\leavevmode
\begin{enumerate}[(a)]
\item $\P$ is $L$--Cockcroft,
\item the map 
$\pi_1(\Sq(\P),1) \otimes_L \Z \rightarrow \bigoplus_{r \in \curlyR} \Z (G/L)$
in \eqref{Cft_ses} is the zero map.
\end{enumerate}
\end{cor}


\begin{thebibliography}{99}
\bibitem{An} D.J. Anick, On the homology of associative algebras.  \textit{Trans. Amer. Math. Soc. } 
\textbf{296}  (1986) 641-659.
\bibitem{BrBook} K.S. Brown, \textit{Cohomology of Groups.}  Graduate Texts in Math. 87, Springer-Verlag New York (1982).
\bibitem{BrHu} R. Brown and J. Huebschmann, Identities among relations. In {\em Low Dimensional Topology}, 
R. Brown and T.L. Thickstun ({\em eds.}), London Math. Soc. Lect. Notes 48, Cambridge
University Press (1982) 153-202.
\bibitem{BrGi} R. Brown and N.D. Gilbert,  Automorphism structures for crossed
modules and algebraic models of 3-types. \textit{Proc. London Math. Soc.} (3) \textbf{59} (1989) 51-73.
\bibitem{RBr} R. Brown, Possible connections between whiskered categories and groupoids, Leibniz algebras, automorphism structures and local-to-global questions. \textit{J. Homotopy Relat. Struct.}, \textbf{1(1)} (2010) 1-13.
\bibitem{Coh} D.E. Cohen, String rewriting -- a survey for group theorists.  In {\em Geometric Group Theory},
G.A. Niblo and M.A. Roller ({\em eds.}), London Math. Soc. Lect. Notes 181 , Cambridge
University Press (1993) 37-47.
\bibitem{CrOt} R. Cremanns and F. Otto, For groups the property of having finite derivation type is equivalent to the
homological finiteness condition $FP_3$.  \textit{J. Symb. Comput.} \textbf{22} (1996) 155-177.
\bibitem{Dy} M.N. Dyer, Crossed modules and $\Pi_2$ homotopy modules.  In {\em Two-dimensional Homotopy and
Combinatorial Group Theory}, C. Hog-Angeloni {\em et al.} (eds)  London Math. Soc. Lect. Notes 197 , Cambridge
University Press (1993) 125-156.
\bibitem{Gi1} N.D. Gilbert, Monoid presentations and associated groupoids. 
\textit{Internat. J. Algebra Comput.} \textbf{8} (1998) 141-152.
\bibitem{GS1} V.S. Guba and M.V. Sapir, {\em Diagram groups.} Memoirs Amer. Math. Soc. 130 (1997) 1-117.
\bibitem{GS2} V.S. Guba and M.V. Sapir, Diagram groups and directed $2$--complexes: homotopy and homology.
\textit{J. Pure Appl. Algebra} \textbf{205} (2006) 1-47.
\bibitem{HiBook} P.J. Higgins, {\em Notes on categories and
groupoids}. Van Nostrand Reinhold Math. Stud. 32 (1971).  Reprinted
electronically at  www.tac.mta.co/tac/reprints/articles/7/7tr7.pdf\,.
\bibitem{Ki} V. Kilibarda, On the algebra of semigroup diagrams. \textit{Internat. J. Algebra Comput.} 
\textbf{7} (1997) 313-338.
\bibitem{OK} F. Otto, and Y. Kobayashi, Properties of monoids that are presented by finite convergent string-rewriting systems---a survey.  
In {\em Advances in algorithms, languages, and complexity}, Kluwer Acad. Publ., (1997) 225-266.
\bibitem{Pr1} S.J. Pride, Low-dimensional homotopy theory for monoids I. 
\textit{Internat. J. Algebra Comput.} \textbf{5} (1995) 631-649.
\bibitem{Pr2} S.J. Pride, Low-dimensional homotopy theory for monoids II: groups.
\textit{Glasg. Math. J.} \textbf{41}(1999) 1--11.
\bibitem{Sq} C.C. Squier, Word problems and a homological finiteness condition for monoids. 
\textit{J. Pure Appl. Algebra} \textbf{49} (1987) 201-216.
\bibitem{SOK} C.C. Squier, F. Otto, and Y. Kobayashi, A finiteness condition for rewriting systems. 
\textit{Theoret. Comput. Sci.} \textbf{131} (1994) 271-294.
\bibitem{Wh2} J.H.C. Whitehead, Combinatorial homotopy II. \textit{Bull. Amer. Math. Soc.} \textbf{55} (1949) 453-496.
\end{thebibliography}
\end{document}